\documentclass[11pt]{amsart}
\usepackage{amsmath,amssymb,amsthm,a4wide,url,comment}
\usepackage{graphicx}
\newtheorem{theorem}{Theorem}[section]  % Standard theorem environment
\newtheorem{lemma}[theorem]{Lemma} 
 
\newtheorem{proposition}[theorem]{Proposition}

\newtheorem{corollary}[theorem]{Corollary}
\theoremstyle{definition}

\newtheorem{remark}[theorem]{Remark}
\newtheorem*{remark*}{Remark}
\newcommand{\Z}{\mathbb{Z}}

\newcommand{\C}{\mathbb{C}}
\newcommand{\Q}{\mathbb{Q}}

\newcommand{\be}{\mathbf{e}}
\newcommand{\strutd}{ \raisebox{-1mm}{\includegraphics*[width=7mm]{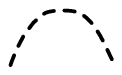}} }
\newcommand{\Aoned}{\raisebox{-1mm}{\includegraphics*[width=3mm]{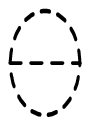}}} \newcommand{\Atwod}{\raisebox{-2mm}{\includegraphics*[width=5mm]{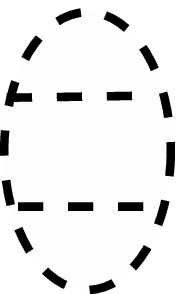}}} \newcommand{\Athreed}{\raisebox{-3mm}{\includegraphics*[width=5mm]{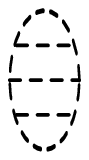}}} 
 
\newcommand{\Donetwod}{\raisebox{-2mm}{\includegraphics*[width=6mm]{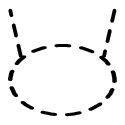}}}
\newcommand{\Donefourd}{\raisebox{-2mm}{\includegraphics*[width=6mm]{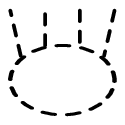}}}
\newcommand{\Donesixd}{\raisebox{-2mm}{\includegraphics*[width=6mm]{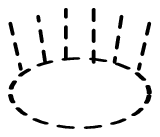}}}
\newcommand{\Dtwotwod}{\raisebox{-2mm}{\includegraphics*[width=4mm]{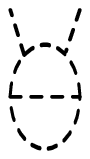}}}
\newcommand{\Dtwofourd}{\raisebox{-2mm}{\includegraphics*[width=6mm]{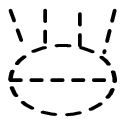}}}
\newcommand{\Dthreetwod}{\raisebox{-2mm}{\includegraphics*[width=4mm]{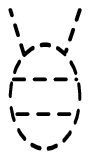}}}
\newcommand{\pairingcommaa}{\, \raisebox{-1mm}{,} \,}
\newcommand{\pairingcommab}{\, \raisebox{-3mm}{,} \,}
\newcommand{\pairingperiod}{\, \raisebox{-3mm}{.} \,}

 \makeatletter
    
    \@addtoreset{equation}{section}
  \makeatother
\title[LMO invariant constrraint for surgery problems]{On LMO invariant constraints for cosmetic surgery and other surgery problems for knots in $S^{3}$}
\author{Tetsuya Ito}
\address{Department of Mathematics, Graduate School of Science, Osaka University \\ 1-1 Machikaneyama Toyonaka, Osaka 560-0043, JAPAN}
\email{tetito@math.sci.osaka-u.ac.jp}
\urladdr{http://www.math.sci.osaka-u.ac.jp/~tetito/}

\begin{document}

\begin{abstract}
We use the LMO invariant to find constraints for a knot to admit a purely or reflectively cosmetic surgery. We also get a constraint for knots to admit a Lens space surgery, and some information for characterizing slopes.
\end{abstract}

\maketitle

\section{Introduction}

%\begin{figure}[htbp]
%\end{figure}
%\setcounter{figure}{0}\vspace{-0.7cm}
For a knot $K$ in $S^{3}$ and $r =p/q \in \Q \cup \{\infty\}$, let $S^{3}(K,r)$ be the oriented closed 3-manifold obtained by the Dehn surgery on $K$ along the slope $r$.
We denote the 3-manifold $M$ with opposite orientation by $-M$, and we write $M \cong M'$ if two 3-manifolds are homeomorphic by an orientation preserving homeomorphism.

Two Dehn surgeries along a knot $K$ with different slopes $r$ and $r'$ are \emph{purely cosmetic} if $S^{3}(K,r) \cong S^{3}(K,r')$ and \emph{reflecitively cosmetic} (or, \emph{chirally cosmetic}) if $S^{3}(K,r) \cong -S^{3}(K,r')$.

A famous cosmetic surgery conjecture (for a knot in $S^{3}$) \cite[Problem 1.81]{kir} states that a non-trivial knot $K$ does not admit purely cosmetic surgeries. The case one of the slope $r=\infty$ was shown in \cite{gl}, which particularly says that two knots in $S^{3}$ are equivalent if and only if they have the homeomorphic complements.

There are various constraints for two Dehn surgeries are purely cosmetic. Among them, using Heegaard Floer homology theory in \cite[Theorem 1.2]{nw} the following strong restrictions are shown.
\begin{equation}
\label{eqn:NW} 
\begin{split}
&
\mbox{If } S^{3}(K,p\slash q) \cong S^{3}(K,p'\slash q'), \mbox{ then } p' \slash q'=\pm p\slash q \mbox{ and } q^{2} \equiv -1  \pmod p.\\ &\mbox{In particular, } L(p,q) \cong L(p',q')
\end{split}
\end{equation}

The cosmetic surgery conjecture can be regarded as a statement saying that when we fix a knot $K$ then the Dehn surgery gives an injective map
$S^{3}(K,\ast):\{\mbox{Slopes}\}=\Q \cup \{\infty\} \rightarrow \{\mbox{(oriented) 3-manifolds}\}$
except the case $K$ is the unknot. In this point of view, it is natural to ask the injectivity of the Dehn surgery map when we fix a slope; Is the Dehn surgery map
$S^{3}(\ast,r):\{\mbox{Knots}\} \rightarrow \{\mbox{(oriented) 3-manifolds}\}$ injective ?
A slope $r$ is a called \emph{characterizing slope} of $K$ if the answer is affirmative, that is, $S^{3}(K,r) \cong S^{3}(K',r)$ implies $K = K'$.

Compared with purely cosmetic surgeries, a situation for characterizing slopes is more complicated. There are various examples of non-characterizing slopes. Among them, in \cite{bamo} hyperbolic knots with infinitely many (integral) non-characterizing slopes are given.
On the other hand, if $K$ is the unknot \cite{kmos,os2}, trefoil, or the figure-eight knot \cite{os1} then all the slopes are characterizing. Moreover, if $K$ is a torus knot, a slope $r$ is characterizing provided $r$ is sufficiently large \cite[Theorem 1.3]{nz}, whereas some small slopes are not characterizing.

In this paper we use the LMO invariant to study a structure of Dehn surgery along knots. We obtain various constraints for a knot to admit a purely or reflectively cosmetic surgery, or, a slope $r$ to be characterizing.

The LMO invariant is an invariant of closed oriented 3-manifolds which takes value in certain graded algebra $\mathcal{A}(\emptyset)$. The degree one part $\lambda_1$ of the LMO invariant is equal to the Casson-Walker invariant \cite{lmmo} that satisfies the following surgery formula
\begin{equation}
\label{eqn:LMO1} \lambda_{1}(S^{3}(K,p \slash q))=\frac{1}{2}a_{2}(K) \frac{q}{p} + \lambda_{1}(L(p,q)).
\end{equation}
Here $a_{2}(K)$ denotes the coefficients of $z^{2}$ of the Conway polynomial, and $L(p,q) = S^{3}(\mathsf{Unknot},p\slash q)$ denote the $(p,q)$-Lens space.

Using (\ref{eqn:NW}) and the surgery formula (\ref{eqn:LMO1}) we immediately get the following constraint for cosmetic surgery and characterizing slopes. (In \cite{boli}, this is proved without using (\ref{eqn:NW}) --  instead they used Casson-Gordon invariant to get an additional constraint.)

\begin{theorem}\cite[Proposition 5.1]{boli}
\label{theorem:LMO1}
Let $K$ and $K'$ be knots in $S^{3}$ and $r,r' \in \Q \setminus \{0\}$ with $r\neq r'$.
\begin{enumerate}
\item[(i)] If $S^{3}(K,r)\cong S^{3}(K',r)$ then $a_{2}(K)=a_{2}(K')$.
\item[(ii)] If $S^{3}(K,r)\cong S^{3}(K,r')$ then $a_{2}(K)=0 \ (=a_{2}({\sf Unknot}))$.
\end{enumerate}
\end{theorem}

Our purpose is to get further constraints that generalize Theorem \ref{theorem:LMO1} by looking at higher order part of the LMO invariants.

Two knots $K$ and $K'$ are called \emph{$C_{n+1}$-equivalent} if $v(K)=v(K')$ for all finite type invariant $v$ whose degree is less than or equal to $n$. 
A knot which is $C_{n+1}$ equivalent to the unknot is called a \emph{$C_{n+1}$-trivial} knot.
In \cite{gu,ha} it is shown that two knots are $C_{n+1}$-equivalent if and only if they are moved each other by certain local moves called $C_{n+1}$-moves. 

In this terminology, Theorem \ref{theorem:LMO1} can be understood that Dehn surgery characterizes a knot or a slope \emph{up to $C_3$-equivalence}: (i) says that if Dehn surgeries of two knots $K$ and $K'$ along the same slope are homeomorphic then $K$ and $K'$ are $C_{3}$-equivalent, and (ii) says that the cosmetic surgery conjecture is true unless $K$ is $C_3$-trivial.

In \cite{bl} Bar-Natan and Lawrence gave a rational surgery formula of the LMO invariant. First we write down a rational surgery formula for degree two and three part of the (primitive) LMO invariants of $S^{3}(K,r)$.

\begin{theorem}[Surgery formula for $\lambda_2$ and $\lambda_3$]
\label{theorem:LMO23}
Let $K$ be a knot in $S^{3}$.
\begin{eqnarray*}
\lambda_2(S^{3}(K,p/q))\!\!&\!\! = \!&\!\!\!
\left(\!v_2(K)^{2} + \frac{1}{24}v_2(K)+\frac{5}{2}v_{4}(K)\!\right)\frac{q^{2}}{p^{2}}-v_3(K)\frac{q}{p} + \frac{v_2(K)}{24}\left( \frac{1}{p^{2}}-1\right)\\
   & & \hspace{0.4cm}+ \lambda_{2}(L(p,q))\\
& \!\!=\!&\!\!\! \left(\!\frac{7a_2(K)^2-a_2(K)-10a_{4}(K)}{8}\! \right)\frac{q^{2}}{p^{2}} -v_{3}(K)\frac{q}{p} + \frac{a_{2}(K)}{48}
\left(1-\frac{1}{p^{2}}\right)\\
 & & \hspace{0.4cm}+ \lambda_{2}(L(p,q))\\
\end{eqnarray*}
\begin{eqnarray*}
\lambda_{3}(S^{3}(K,p/q)) \!
&\!\! = \!& \!
-\left(\!\frac{35}{4}v_6(K)\!+\!\frac{5}{24}v_4(K)\!+\!10v_2(K)v_4(K)\!+\!\frac{4}{3}v_2(K)^{3}\!+\! \frac{1}{12}v_2(K)^{2}\!\right)\frac{q^{3}}{p^{3}}\\
& &\! - \left( \frac{5}{24}v_4(K) +\frac{1}{288}v_2(K)+\frac{1}{12}v_2(K)^{2}\right)\frac{q}{p^{3}} \\
& &\! + \left(\frac{5}{2}v_{5}(K)+ 2v_3(K)v_2(K)+\frac{1}{24}v_3(K)\right)\frac{q^2}{p^2} + \frac{v_3(K)}{24}\left(\frac{1}{p^2}-1\right)\\
& &\! - \left(w_4(K)-\frac{1}{12}v_2(K)^{2}-\frac{1}{288}v_2(K)-\frac{5}{24}v_4(K)\right)\frac{q}{p}  +\lambda_{3}(L(p,q))
\end{eqnarray*}

\end{theorem}

Here $v_2(K),v_3(K),v_4(K),w_4(K),v_5(K)$ and $v_6(K)$ are certain canonical finite type invariant of the knot $K$ (see Section \ref{section:LMO} for details -- as we will see in Lemma \ref{lemma:AJ}, except $v_5$ they are determined by the Alexander and the Jones polynomial). Also, $a_{2n}(K)$ is the coefficient of $z^{2n}$ in $\nabla_{K}(z)$, the Conway polynomial of $K$.

The degree two part of the LMO invariant (combined with (\ref{eqn:NW})) gives rise to the followings.

\begin{corollary}
\label{corollary:LMO2}
Let $K$ and $K'$ be knots in $S^{3}$, and $r,r' \in \Q \setminus \{0\}$ with $r \neq r'$.
\begin{enumerate}
\item[(i)] If $S^{3}(K,r) \cong S^{3}(K,r')$ then $v_3(K)=0$. 
\item[(ii)] If $S^{3}(K,r) \cong -S^{3}(K,-r)$ then $v_3(K)=0$. 
\item[(iii)] If $S^{3}(K,r) \cong -S^{3}(K,r')$ for $r' \neq \pm r$ then either
\begin{itemize}
\item[(iii-a)] $v_{3}(K)=0$, or,
\item[(iii-b)] $v_{3}(K) \neq 0$ and 
$\displaystyle \frac{rr'}{r+r'} = \frac{7a_2(K)^2-a_2(K)-10a_{4}(K)}{8v_{3}(K)}. $
\end{itemize}
\item[(iv)] If $S^{3}(K,r) \cong S^{3}(K',r)$ then either
\begin{itemize}
\item[(iv-a)] $a_{4}(K)=a_{4}(K')$, $v_{3}(K)=v_{3}(K')$, or,
\item[(iv-b)] $a_{4}(K) \neq a_{4}(K')$, $v_{3}(K)\neq v_{3}(K')$, and 
$\displaystyle  r = \frac{5(a_{4}(K)-a_{4}(K'))}{4(v_{3}(K)-v_{3}(K'))}. $
\end{itemize}
\end{enumerate}
\end{corollary}

(i) was proven in \cite{iw} by a similar argument using Lescop's surgery formula of the Kontsevich-Kuperberg-Thurston invariant \cite{ko2,kt} (see Remark \ref{remark:KKT}). 
 
We note that the degree two part gives the following constraint for a knot to admit a Lens space surgery.

\begin{corollary}
\label{corollary:Lens}
If $S^{3}(K,p\slash q)$ is a Lens space, then 
\[ \left( \frac{7a_2(K)^2-a_2(K)-10a_{4}(K)}{8} \right)\frac{q^{2}}{p^{2}} -v_{3}(K)\frac{q}{p} + \frac{a_{2}(K)}{48}
\left(1-\frac{1}{p^{2}}\right) =0.\]
\end{corollary}

By cyclic surgery theorem \cite{cgls}, if $K$ is not a torus knot, then $q=1$ hence we get
\begin{equation}
\label{eqn:Lens-obstruction}
a_{2}(K)p^{2} -48v_{3}(K)p+\left(42a_2(K)^2-7a_2(K)-60a_{4}(K) \right) =0.
\end{equation}

Combined with the fact that $a_{2}(K),4v_{3}(K)$ and $a_{4}(K)$ are integers, (\ref{eqn:Lens-obstruction}) brings some interesting informations. For example, a non-torus knot $K$ admitting lens surgery, $576v_{3}(K)^{2}-a_{2}(K)(42a_2(K)^2-7a_2(K)-60a_{4}(K))$ is a square number. If a non-torus knot $K$ admits more than one Lens space surgeries, the surgery slopes are successive integers \cite[Corollary 1]{cgls} so such a knot has $a_{2}(K) \neq \pm 1$.

The formula of degree three part is more complicated. Fortunately, as for cosmetic surgery, using (\ref{eqn:NW}) we get the following simple constraints.

\begin{corollary}
\label{corollary:LMO3}
Let $K$ and $K'$ be a knot in $S^{3}$ and $r= p\slash q \in \Q \setminus\{0\}$.
\begin{enumerate}
\item[(i)] If $S^{3}(K,r) \cong S^{3}(K,r')$ for $r'\neq r$, then 
\[ p^{2}(24w_{4}(K)-5v_{4}(K)) + 5v_{4}(K) + q^{2}(210v_{6}(K)+5v_{4}(K)) = 0.\]
\item[(ii)] If $S^{3}(K,r) \cong -S^{3}(K,-r)$, then $v_{5}(K)=0$.
\end{enumerate}
\end{corollary}
Corollary \ref{corollary:LMO3} (i) leads to the following.
\begin{corollary}
\label{cor:c11}
The cosmetic surgery conjecture is true for all knots with less than or equal to 11 crossings, possibly except $10_{118}$.
\end{corollary}

Using higher degree part of the LMO invariant, adding suitable $C_n$-equivalence assumptions we prove the following more direct generalizations of Theorem \ref{theorem:LMO1}.

\begin{theorem}
\label{theorem:LMOhigh1}
Let $K$ and $K'$ be a knot in $S^{3}$ and $r, r' \in \Q \setminus\{0\}$ with $r \neq r'$.
\begin{enumerate}
\item[(i)] Assume that $K$ and $K'$ are $C_{2m+2}$-equivalent.
If $S^{3}(K,r) \cong S^{3}(K',r)$ then $a_{2m+2}(K) = a_{2m+2}(K')$.
\item[(ii)] Assume that $K$ is $C_{4m+2}$-trivial.
If $S^{3}(K,r) \cong S^{3}(K,r')$ then $a_{4m+2}(K) = 0$.
\end{enumerate}
\end{theorem}

We say that $K$ and $K'$ are \emph{odd $C_{n+1}$}-equivalent if $v(K)=v(K')$ for all \emph{odd degree} canonical finite type invariant $v$ of degree $\leq n$. We call a knot is \emph{odd $C_{n+1}$-trivial} if it is odd $C_{n+1}$-equivalent to the unknot.

Let $(Z^{\sigma}(K) \sqcup \Omega^{-1})_{e,k}$ denotes the bigrading $(e,k)$ of the Kontsevich invariant of $K$, normalized so that the unknot takes value $1$. Let $K_{e,k}$ denotes the kernel of the Aarhus integration (diagram pairing) $\langle \ast , \strutd^{\frac{k}{2}}\rangle: \mathcal{B}_{e,k} \rightarrow \mathcal{A}(\emptyset)_{e + \frac{k}{2}}$ (See Section \ref{section:LMO}).

\begin{theorem}
\label{theorem:LMOhigh2}
Let $K$ and $K'$ be a knot in $S^{3}$ and $r \in \Q \setminus\{0\}$.
\begin{enumerate}
\item[(i)] Assume that $K$ and $K'$ are $C_{2m+1}$-equivalent.
If $S^{3}(K,r) \cong S^{3}(K',r)$ and $a_{2m+2}(K) = a_{2m+2}(K')$, then $(Z^{\sigma}(K) \sqcup \Omega^{-1})_{1,2m}- (Z^{\sigma}(K) \sqcup \Omega^{-1})_{1,2m} \in K_{1,2m}$.
\item[(ii)]  Assume that $K$ is odd $C_{4m+1}$-trivial. 
If $S^{3}(K,r) \cong - S^{3}(K,-r)$ then $(Z^{\sigma}(K) \sqcup \Omega^{-1})_{1,4m} \in K_{1,4m}$. 
\item[(iii)] Assume that $K$ is odd $C_{4m+3}$-trivial.
If $S^{3}(K,r) \cong \pm S^{3}(K,-r)$ then $(Z^{\sigma}(K) \sqcup \Omega^{-1})_{1,4m+2} \in K_{1,4m+2}$. 
\end{enumerate}
\end{theorem}

As a corollary, we prove a vanishing of certain finite type invariants that come from colored Jones polynomial (Quantum $\mathfrak{sl}_2$ invariant). Let $V_{n}(K;t)$ be the $n$-colored Jones polynomial, normalized $V_{n}(\textsf{Unknot};t) = 1$. The colored Jones polynomials have the following expansion called the \emph{loop expansion}, or \emph{Melvin-Morton expansion} \cite{mm}.
\[ V_{n}(K;e^{h}) = \sum_{e\geq 0}  \Bigl(\sum_{k\geq 0} j_{e,k}(K) (nh)^{k}\Bigr) h^{e}. \] 
Here the coefficient $j_{e,k}(K) \in \Q$ is a canonical finite type invariant of degree $e+k$.

\begin{corollary}
\label{corollary:MM}
Let $K$ and $K'$ be a knot in $S^{3}$ and $r \in \Q \setminus\{0\}$.
\begin{enumerate}
\item[(i)] Assume that $K$ and $K'$ are $C_{2m+1}$-equivalent.
If $S^{3}(K,r) \cong S^{3}(K',r)$ and $a_{2m+2}(K) = a_{2m+2}(K')$, then $j_{1,2m}(K)=j_{1,2m}(K')$.
\item[(iii)]  Assume that $K$ is odd $C_{4m+1}$-trivial.
If $S^{3}(K,r) \cong - S^{3}(K,-r)$ then $j_{1,4m}(K)=0$. 
\item[(ii)] Assume that $K$ is odd $C_{4m+3}$-trivial.
If $S^{3}(K,r) \cong \pm S^{3}(K,-r)$ then $j_{1,4m+2}(K)=0$.
\end{enumerate}
\end{corollary}

For a canonical finite type invariant $v$ of degree $n$, $v(\mbox{mirror of }K) = (-1)^{n}v(K)$. Thus for an amphicheiral knot $K$, $v(K)=0$ for all  \emph{odd} degree  canonical finite type invariant $v$. It is conjectured that converse is true (this is related to a more familiar conjecture that finite type invariants do not detect the orientation of knots \cite[Problem 1.89]{kir}). Corollary \ref{corollary:LMO2} (i), Corollary \ref{corollary:LMO3} (ii), Corollary \ref{corollary:MM} (ii),(iii) says that if $S^{3}(K,r) \cong -S^{3}(K,-r)$ then various canonical odd degree finite type invariants of $K$ vanish. Thus, they bring a supporting evidence for an affirmative answer to the following question.
\[ \mbox{\emph{If} } S^{3}(K,r) \cong -S^{3}(K,-r) \mbox{\emph{ for some }} r \neq 0,\infty, \mbox{\emph{ then is }} K \mbox{\emph{ amphicheiral?}}\]

\section*{Acknowledgements}
The author would like to thank K. Ichihara for stimulating talk and discussion which inspires the author to work on this subject, and to thank to Z. Wu for pointing out inaccuracy of reference in the first version of the paper. This work was partially supported by JSPS KAKENHI 15K17540,16H02145.

\section{LMO invariants and rational surgery formula}
\label{section:LMO}

In this section we briefly review the basics of Kontsevich and LMO invariants. We use Aarhus integral construction of the LMO invariant developed \cite{bgrt1,bgrt2,bgrt3} and a rational surgery formula of the LMO invariant due to Bar-Natan and Lawrence \cite{bl}. For basics of the Kontsevich and the LMO invariants we refer \cite{oht}.

\subsection{Open Jacobi diagrams}

An \emph{(open) Jacobi diagram} or \emph{(vertex-oriented) uni-trivalent graph} is a graph $D$ whose vertex is either univalent or trivalent, such that at each trivalent vertex $v$ a cyclic order on three edges around $v$ is equipped. 
The \emph{degree} of $D$ is the half of the number of vertices.
We will often call a univalent vertex a \emph{leg}, and denote the number of the legs of a Jacobi diagram $D$ by $k(D)$.
For a Jacobi diagram $D$, let $e(D)=-\chi(D)$ be the minus of the euler characteristic of $D$. We call $e(D)$ the \emph{euler degree} of $D$. 
Then $\deg(D)=e(D)+k(D)$.

Let $\mathcal{B}$ (resp. $\mathcal{A}(\emptyset)$) be the vector space over $\C$ spanned by Jacobi diagrams (resp. Jacobi diagrams without univalent vertex), modulo the AS and IHX relations given in Figure \ref{fig:IHX}. 

\begin{figure}[htbp]
\includegraphics*[width=100mm]{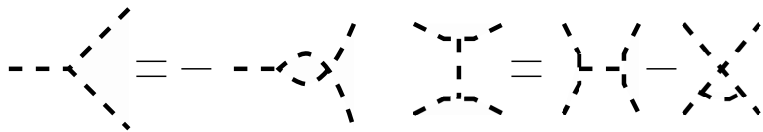}
\caption{The AS and IHX relation: we understand that at each trivalent vertex, cyclic order is defined by counter-closkcwise direction.} 
\label{fig:IHX}
\end{figure} 

By taking the disjoint union $\sqcup$ as the product, both $\mathcal{B}$ and $\mathcal{A}(\emptyset)$  have the structure of graded algebras. Since the IHX and the AS relations and the disjoint union product respect both $k(D)$ and $e(D)$, we view $\mathcal{B}$ as a bi-graded algebra. 
For $X \in \mathcal{B}$ we denote by $X_{e,k}$ the part of $X$ whose bigrading is $(e,k)$.
Strictly speaking, we will use the completion of $\mathcal{B}$ and $\mathcal{A}(\emptyset)$ with respect to degrees which we denote by the same symbol $\mathcal{B}$ and $\mathcal{A}(\emptyset)$ by abuse of notations.

Let $\exp_{\sqcup}:\mathcal{B} \rightarrow \mathcal{B}$ (or, $\mathcal{A}(\emptyset) \rightarrow \mathcal{A}(\emptyset)$) be the exponential map with respect to $\sqcup$ product operation, defined by
\[ \exp_{\sqcup}(D) = 1 + D + \frac{1}{2}D \sqcup D + \cdots = \sum_{n=0}^{\infty} \frac{1}{n!}\underbrace{(D \sqcup\cdots \sqcup D)}_{n}. \] 
We will simply denote $\underbrace{(D \sqcup\cdots \sqcup D)}_{n}$ by $D^{n}$. 
%Similarly, we denote by $\log_{\sqcup}$ the logarithm map with respect to $\sqcup$ product operation.

For a Jacobi diagram $C$, let $\partial_{C}:\mathcal{B} \rightarrow \mathcal{B}$ be the differential operator defined by
\[ \partial_{C} (D) =\begin{cases} 0 & k(C)>k(D) \\
\sum(\mbox{glue} \textit{ all } \mbox{the legs of } \Omega \mbox{ to some legs of }D) & k(C)\leq k(D)
\end{cases}
\]
In a similar manner, we define the pairing $\langle C,D\rangle \in\mathcal{A}(\emptyset)$ of $C,D \in \mathcal{B}$ by 
\[ 
\langle C,D \rangle = 
\begin{cases}
0 & k(C) \neq k(D)\\
\sum \mbox{(glue the legs of } C \mbox{ to the legs of } D) & k(C)=k(D)
\end{cases}
\]
Thus $\partial_{C}(D)=\langle C,D\rangle$ if $k(C)=k(D)$.
In both cases, the summation runs all the possible ways to gluing \emph{all} the legs of $C$ to \emph{some} legs of $D$. We denote this summation by using box, as Figure \ref{fig:pairing}.
It is known that $\partial_{C \sqcup C'} = \partial_{C'}\circ\partial_{C}$. Thus if $C \in \mathcal{B}$ is invertible (with respect to $\sqcup$ product) then $\partial_{C}$ is invertible with $\partial_{C}^{-1}=\partial_{C^{-1}}$ (see \cite{bgrt2,blt,bl} for details).

\begin{figure}[htbp]
\includegraphics*[width=75mm]{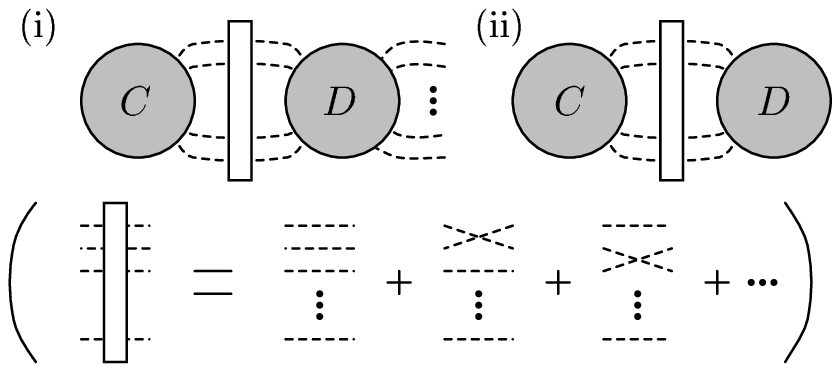}
\caption{(i) Differential operator $\partial_{C}(D)$ (ii) Pairing $\langle C,D \rangle$.} 
\label{fig:pairing}
\end{figure}

Let $b_{2i}$ be the modified Bernoulli numbers, defined by 
\begin{equation}
\label{eqn:modBer} \frac{1}{2}\log \frac{\sinh(x\slash2)}{x \slash 2} = \sum_{i=0}^{\infty} b_{2i} x^{2i} = 1+ \frac{1}{48}x^{2}-\frac{1}{5760}x^{4}+ \frac{1}{362880}x^{6}+ \cdots. 
\end{equation}
For $q \in \Z\setminus\{0\}$, let
\[ \Omega_{q} = \exp_{\sqcup} \Bigl(\sum_{n=1}^{\infty} \frac{b_{2n}}{q^{2n}}
\overbrace{\raisebox{-4.5mm}{\includegraphics*[width=10mm]{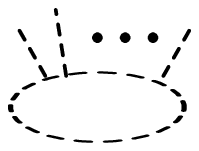}}}^{2n} \;\Bigr) = 1 + \frac{1}{48q^{2}}\Donetwod - \frac{1}{5760q^{4}}\Donefourd + \frac{1}{4608q^{4}}\Donetwod \Donetwod +\cdots. \]

The element $\Omega=\Omega_{1}$ is called the \emph{wheel element}.

\subsection{Wheeled Kontsevich invariant}

The \emph{Kontsevich invariant} $Z(K)$ is an invariant of a framed knot, which takes value in $\mathcal{A}(S^{1})$, the space of Jacobi diagram over $S^{1}$ \cite{ko1,bar}. Throughout the paper, we will always assume that the knot $K$ is zero-framed.
The target space $\mathcal{A}(S^{1})$ is isomorphic to $\mathcal{B}$ as a graded vector space, by Poincar\'e-Birkoff-Witt isomorphism $\chi: \mathcal{B} \rightarrow \mathcal{A}(S^{1})$. Let $\sigma:\mathcal{A}(S^{1}) \rightarrow \mathcal{B}$ be the inverse of $\chi$. In the rest of the paper, we will always view the Kontsevich invariant takes value in $\mathcal{B}$, by defining
\[ Z^{\sigma}(K) = \sigma(Z(K)) \in \mathcal{B} \]
We will denote by $Z^{\sigma}(K)_{e,k}$ the bigrading $(e,k)$ part of the Kontsevich invariant. See \cite{gr} for a topological meaning of this bigrading. 

Let $V_n$ be the vector space spanned by finite type invariants of degree $\leq n$. The Kontsevich invariant gives a map $\mathcal{Z}:(\mathcal{B}_{\deg =n})^{*} \rightarrow V_{n}$, by $\mathcal{Z}(w)(K)= w(Z^{\sigma}(K)_{\deg=n})$. Here $w: \mathcal{B}_{\deg=n} \rightarrow \C$ is an element of $(\mathcal{B}_{\deg =n})^{*}$, the dual space of degree $n$ part of $\mathcal{B}$ and $Z^{\sigma}(K)_{\deg=n} \in \mathcal{B}_{\deg=n}$ denotes the degree $n$ part of $Z^{\sigma}(K)$.
On the other hand, there is a map called \emph{symbol} $\mathsf{Symb}: V_{n} \rightarrow (\mathcal{B}_{\deg = n})^{*}$ (see \cite{bar,bg} for definition). A finite type invariant $v \in V_n$ is called
\begin{itemize}
\item[--] \emph{canonical}, if $v= \mathcal{Z}(\mathsf{Symb}(v))$.
\item[--] \emph{primitive}, if $v(K\# K')=v(K)+v(K')$, which is equivalent to saying that $v$ lies in the image of the primitive subspace of $\mathcal{B}$.
\end{itemize}

The \emph{wheeled Kontsevich invariant} $Z^{\sf Wheel}(K) \in \mathcal{B}$ is a version of the Kontsevich invariant defined as follows.

Let $\partial_{\Omega}= 1 + \frac{1}{48}\partial_{\includegraphics*[width=3mm]{D_12.eps}}+\cdots$ be the differential operator defined by the wheel element $\Omega$. The \emph{wheeling map} $\Upsilon = \chi \circ \partial_{\Omega}:\mathcal{B} \rightarrow \mathcal{A}(S^{1})$ is the composite of $\partial_{\Omega}$
and the Poincar\'e-Birkoff-Witt isomorphism $\chi$.
The wheeling map $\Upsilon$ gives an isomorphism of \emph{algebra} \cite[Wheeling theorem]{blt}, whereas Poincar\'e-Birkoff-Witt isomorphism $\chi$ only gives an isomorphism of vector space.

The wheeled Kontsevich invariant is the image of the Kontsevich invariant under the inverse of the wheeling map $\Upsilon$: 
\[ Z^{\sf Wheel}(K) = \Upsilon^{-1}(Z(K)) = (\partial_{\Omega})^{-1}\circ\sigma(Z(K)) = \partial_{\Omega^{-1}} Z^{\sigma}(K). \]

The wheel element $\Omega$ is equal to the Kontsevich invariant of the unknot \cite[Wheel theorem]{blt}: $Z^{\sigma}({\sf Unknot})=\Omega$. Therefore instead of $Z^{\sigma}$ or $Z^{\sf Wheel}$, it is often useful to use $Z^{\sigma}(K) \sqcup \Omega^{-1}$ since $Z^{\sigma}({\sf Unknot}) \sqcup \Omega^{-1} = 1$.

The Kontsevich invariant is group-like, so $Z^{\sigma}(K) = \exp_{\sqcup}(z^{\sigma}(K))$, where $z^{\sigma}(K)$ denotes the primitive part of $Z^{\sigma}(K)$. We express low degree part of the primitive Kontsevich invariant as
\begin{eqnarray*}
Z^{\sigma}(K) \sqcup \Omega^{-1}&= &\exp_{\sqcup}\Bigl( v_2(K)\Donetwod + v_3(K)\Dtwotwod+ v_{4}(K)\Donefourd + w_4(K) \Dthreetwod \\
& & \hspace{1.8cm}+v_5(K)\Dtwofourd  +v_6(K)\Donesixd+(\mbox{higher degree parts}) \Bigr).
\end{eqnarray*}
Here $v_{2}(K),v_3(K),v_4(K),w_4(K),v_5(K),v_6(K)$ are canonical, primitive finite type invariants of degree $2,3,4,4,5,6$, respectively.

Thus the bigrading $(e,k)$ part of $Z^{\sigma}(K)$ with $e+\frac{k}{2} \leq 3$ are explicitly written by
\begin{eqnarray*}
Z^{\sigma}(K)&\!\!=\!\!&1 +\bigl(v_2(K) + b_2\bigr) \Donetwod + v_3(K)\Dtwotwod+ \frac{1}{2}\bigl(v_2(K) +b_{2}\bigr)^{2} \Donetwod \Donetwod +(v_{4}(K)+b_{4}) \Donefourd\\ 
& &+ w_4(K) \Dthreetwod + v_3(K)(v_2(K)+b_2)\Donetwod \Dtwotwod + v_5(K)  \Dtwofourd + \frac{1}{6}\bigl(v_2(K) +b_{2}\bigr)^{3}\Donetwod \Donetwod\Donetwod \\
& &+\bigl( v_2(K)+b_2\bigr)\bigl( v_4(K)+b_4 \bigr)\Donetwod \Donefourd + \bigl(v_6(K) +b_{6}\bigr)\Donesixd+(\mbox{higher degree parts}).
\end{eqnarray*}
Here $b_{2i}$ denotes the modified Bernoulli numbers given by (\ref{eqn:modBer}).

Except $v_{5}(K)$, these finite type invariants are written by using Conway polynomial and the Jones polynomials.
Let $a_{2i}(K)$ be the coefficient of $z^{2i}$ in the Conway polynomial $\nabla_{K}(z)$ of $K$, and let $j_{n}(K)$ is the coefficient of $h^{n}$ in the Jones polynomial $V_{K}(e^{h})$ of $K$, putting the variable as $t=e^{h}$. Then we have the following (See Section \ref{section:sl2computation} for proof).

\begin{lemma}
\label{lemma:AJ}
\begin{enumerate}
\item $v_{2}(K)=-\frac{1}{2}a_{2}(K)$.
\item $v_{3}(K)=-\frac{1}{24}j_{3}(K)$.
\item $v_{4}(K)=-\frac{1}{2}a_{4}(K)-\frac{1}{24}a_{2}(K)+\frac{1}{4}a_{2}(K)^{2}$.
\item $w_{4}(K)=\frac{1}{96}j_{4}(K) + \frac{3}{32}a_{4}(K) - \frac{9}{2}a_{2}(K)^{2}$.
\item $v_{6}(K)=-\frac{1}{2}a_{6}(K)-\frac{1}{12}a_{4}(K)-\frac{1}{720}a_{2}(K)+\frac{1}{24}a_{2}(K)^{2}+\frac{1}{2}a_{2}(K)a_{4}(K)-\frac{1}{6}a_{2}(K)^{3}$.
\end{enumerate}
\end{lemma}

Since the Jones polynomial is an integer coefficient polynomial, $j_{3}(K) \in 6\Z$ so $4v_{3}(K) \in \Z$. The degree three finite type invariant $v_3$ takes value $-\frac{1}{4}$ for a right-handed trefoil.

In general, the euler degree zero part of the Kontsevich invariant is wirtten by the Alexander polynomial, using the following formula (This is a consequence of Melvin-Morton-Rozansky conjecture \cite{bg}).

\begin{proposition}
\label{proposition:MMR}
Let $-\frac{1}{2} \log \Delta_{K}(e^{x})=\sum_{n=0}^{\infty} d_{2n}(K) x^{2n}$ where $\Delta_{K}(t)$ is the Alexander polynomial of $K$, normalized so that $\Delta_{K}(t)=\Delta_{K}(t^{-1})$, $\Delta_{K}(1)=1$. 
Then the euler degree zero part of the Kontsevich invariant is
\[ \exp_{\sqcup}\Bigl(\sum_{k=0}^{\infty} d_{2k}(K) \overbrace{\raisebox{-4.5mm}{\includegraphics*[width=10mm]{D_12m.eps}}}^{2n} \Bigr).\]
In particular, if $a_2(K)=a_{4}(K)=\cdots =a_{2m}(K)=0$ for some $m\geq0$ then $d_{2}(K)=d_4(K)=\cdots = d_{2m}(K)=0$ and $d_{2m+2}(K)=-\frac{1}{2}a_{2m+2}(K)$.
\end{proposition}

\subsection{LMO invariant and rational surgery formula}

The LMO invariant $\widehat{Z}^{LMO}(M)$ is an invariant of an oriented closed 3-manifold $M$ that takes value in $\mathcal{A}(\emptyset)$.
Here we restrict our attention to the case $M=S^{3}(K,r)$ for $r \neq 0,\infty$. In particular, we will always assume that $M$ is a rational homology sphere.

To make computation simpler, we will use the following simplification. 
Let $\Aoned$ be the theta-shaped Jacobi diagram which generates the degree one part of $\mathcal{A}(\emptyset)$.

Let $\mathcal{A}^{red}$ be the quotient of $\mathcal{A}(\emptyset)$ by the ideal generated by $\Aoned$, and let $\pi: \mathcal{A}(\emptyset) \to \mathcal{A}^{red}$ be the quotient map. 
We call $\pi(\widehat{Z}^{LMO}(M)) \in \mathcal{A}^{red}$ the \emph{reduced LMO invariant} and denote by $Z^{LMO}(M)$.
By abuse of notation, we will simply refer the reduced LMO invariant simply as the LMO invariant. When we change the orientation, the (reduced) LMO invariant changes as
\[ Z^{LMO}_{n}(-M) = (-1)^{n}Z^{LMO}_n(M) \]
where $Z^{LMO}_{n}(M)$ denotes the degree $n$ part of the LMO invariant.

The low degree part of the (reduced) LMO invariant is written by
\[ Z^{LMO}(M)=1 + \lambda_{2}(M) \Atwod + \lambda_{3}(M) \Athreed +(\mbox{higher degree parts}). \]
where $\lambda_2(M), \lambda_3(M) \in \C$ are finite type invariant of rational homology spheres.
In the rest of arguments, unless otherwise specified, we will always work in $\mathcal{A}^{red}$. For example, we will always view the pairing $\langle D,D' \rangle$ so that it takes value in $\mathcal{A}^{red}$, by composing the quotient map $\pi$.

Using the simplification $\Aoned=0$, the (reduced) LMO invariant of the 3-manifold obtained by rational Dehn surgery along a knot $K$ is given by the following simpler formula. Let \strutd be the strut, the Jacobi diagram homeomorphic to the interval. 

\begin{theorem}[Rational surgery formula \cite{bl}]
\label{theorem:rsurgery}
Let $K$ be a knot in $S^{3}$. Then the (reduced) LMO invariant of $p/q$-surgery along $K$ is given by
\[ Z^{LMO}(S^{3}(K,p/q))= \Bigl\langle Z^{\sf Wheel}(K) \sqcup \Omega_{q}  \pairingcommaa  \exp_{\sqcup}( -\frac{q}{2p}\raisebox{-1mm}{\includegraphics*[width=7mm]{strut.eps}})  \Bigr\rangle. \]
\end{theorem}

As an application of the rational surgery formula above, in \cite[Proposition 5.1]{bl} it is shown that the (reduced) LMO invariant of the lens space $L(p,q)$ is given by the
\begin{equation}
\label{eqn:LMO-lens}
Z^{LMO}(L(p,q)) = \langle \Omega, \Omega^{-1} \sqcup \Omega_{p}\rangle
\end{equation}

Thus the (reduced) LMO invariant of Lens space only depend on $p$.
In particular,
\begin{equation}
\label{eqn:LMO23-Lens}
\lambda_{2}(L(p,q))= \frac{1}{24}\left( \frac{1}{48p^{2}}-\frac{1}{48} \right),\quad Z^{LMO}_{2m+1}(L(p,q))= 0 \ (m\in \Z).
\end{equation}

\section{Proof of Theorems}

First of all we determine which part of the Kontsevich invariant contributes to the degree $n$ part of the (reduced) LMO invariants.

\begin{proposition}
\label{prop:degree}
The degree $n$ part of the LMO invariant for $S^{3}(K,p\slash q)$ is determined by the slope $p\slash q$ and the bigrading $(e,k)$ part of $Z^{\sigma}(K)_{e,k}$ with $e+ \frac{k}{2} \leq n$.
\end{proposition}
\begin{proof}

By definition of the pairing, for $D \in \mathcal{B}_{e,k}$ we have 
$\langle D, \exp_{\sqcup}(-\frac{q}{2p} \strutd ) \rangle \in \mathcal{A}(\emptyset)_{e+\frac{k}{2}}$. 
Thus the degree $n$ part of the LMO invariant of $S^{3}(K,p\slash q)$ is determined by the bigrading $(e,k)$ part of $Z^{\sf Wheel}(K) \sqcup \Omega_{q}$, with $e+\frac{k}{2}=n$.

The bigrading $(e,k)$ part of $Z^{\sf Wheel}(K) \sqcup \Omega_{q}$ is determined by the bigrading $(e',k')$ part of $Z^{\sf Wheel}(K)$ with $(e',k') \in \{(e,k), (e,k-2),(e,k-4),\ldots \}$.
Also, by definition of $\partial_{D}$, if $D \in  \mathcal{B}_{e,k}$ and $D' \in \mathcal{B}_{e',k'}$, $\partial_{D}(D') \in \mathcal{B}_{e'+e+k,k'-k}$. This shows that the bigrading $(e,k)$ part of $Z^{\sf Wheel}(K)$ is determined by the bigrading $(e',k')$ part of $Z^{\sf Wheel}(K)$ with $(e',k') \in \{(e,k), (e-2,k+2),(e-4,k+4),\ldots \}$.

These observations show that the degree $n$ part of the LMO invariant of $S^{3}(K,p\slash q)$ is determined by the bigrading $(e,k)$ part of $Z^{\sigma}(K)_{e,k}$ with $e+ \frac{k}{2} \leq n$ (and the surgery slope $p\slash q$).
\end{proof}

\begin{proof}[Proof of Theorem \ref{theorem:LMO23}]
By Proposition \ref{prop:degree}, to compute the degree 2 and 3 part of the LMO invariant for $S^{3}(K,p\slash q)$, it is sufficient to consider the bigraging $(e,k)$ part of $Z^{\sigma}(K)$ for $e + \frac{k}{2} \leq 3$. As we have already seen, this is given by
\begin{eqnarray*}
Z^{\sigma}(K)&\!\!=\!\!&1 +\bigl(v_2(K) + b_2\bigr) \Donetwod + v_3(K)\Dtwotwod+ \frac{1}{2}\bigl(v_2(K) +b_{2}\bigr)^{2} \Donetwod \Donetwod +(v_{4}(K)+b_{4}) \Donefourd\\ 
& &+ w_4(K) \Dthreetwod + v_3(K)(v_2(K)+b_2)\Donetwod \Dtwotwod + v_5(K)  \Dtwofourd + \frac{1}{6}\bigl(v_2(K) +b_{2}\bigr)^{3}\Donetwod \Donetwod\Donetwod \\
& &+\bigl( v_2(K)+b_2\bigr)\bigl( v_4(K)+b_4 \bigr)\Donetwod \Donefourd + \bigl(v_6(K) +b_{6}\bigr)\Donesixd+(\mbox{higher degree parts}).
\end{eqnarray*}
Here $b_2=\frac{1}{48}, b_{4}=-\frac{1}{5760}, b_{6}= \frac{1}{362880}$ are modified Bernoulli numbers.
Since 
\[ \begin{cases}
\partial_{\Omega^{-1}} \left( \Donetwod \right) = \Donetwod-2b_2\Atwod, \ 
\partial_{\Omega^{-1}} \left(\Dtwotwod \right) =\Dtwotwod -2 b_2\Athreed, \\
\partial_{\Omega^{-1}}\left(\Donetwod \Donetwod \right)= \Donetwod \Donetwod - 8b_2\Dthreetwod +\mbox{(other parts),}\\
\partial_{\Omega^{-1}}\left(\Donefourd\right)= \Donefourd - 10b_2\Dthreetwod + \mbox{(other parts)},
\end{cases}
\]
the wheeled Kontsevich invariant is given by
\begin{align*}
Z^{\sf Wheel}(K)&= 1 +\bigl(v_2(K) + b_2\bigr)\Donetwod
 + v_3(K)\Dtwotwod + \frac{1}{2}\bigl(v_2(K)+b_{2}\bigr)^{2}\Donetwod \Donetwod +(v_{4}(K)+b_{4}) \Donefourd \\
& \hspace{0.4cm} + w_4(K)\Dthreetwod + v_3(K)(v_2(K)+b_2) \Donetwod\Dtwotwod + v_5(K)\Dtwofourd + \frac{1}{6}\bigl(v_2(K) +b_{2}\bigr)^{3}\Donetwod \Donetwod \Donetwod  \\
& \hspace{0.4cm} + \bigl( v_2(K)+b_2\bigr)\bigl( v_4(K)+b_4 \bigr)\Donetwod \Donefourd +\bigl(v_6(K)+b_{6}\bigr)\Donesixd +\bigr(-2b_2 \bigr)\bigl(v_2(K)+ b_2\bigr)\Atwod \\
& \hspace{0.4cm}+ (-2b_2v_3(K))\Athreed +  \bigl(-8b_2\frac{1}{2}\bigl(v_2(K) +b_{2}\bigr)^{2}-10b_2v_4(K)\bigr)\Dthreetwod+ (\mbox{other parts}).
\end{align*}

Thus $Z^{\sf Wheel}(K)\sqcup \Omega_{q}$ is equal to
\begin{eqnarray*}
& & Z^{\sf Wheel}(K)\sqcup \Omega_{q}\\
&=& 1 +\bigl(v_2(K) + b_2 + \frac{b_2}{q^{2}}\bigr)\Donetwod + v_3\Dtwotwod \\
& & + \Bigl\{ \frac{1}{2}\bigl(v_2(K) +b_{2}\bigr)^{2} + (v_2(K)+b_2)\frac{b_2}{q^{2}}+ \frac{1}{2}\frac{b_2^2}{q^{4}}\Bigr\} \Donetwod \Donetwod \\
& &  + (v_{4}(K)+b_{4}+\frac{b_{4}}{q^{4}})\Donefourd + \bigl(w_4(K)-4b_2\bigl(v_2(K) +b_{2}\bigr)^{2}-10b_2v_4(K)\bigr)\Dthreetwod\\
& & + \bigl(v_3(K)(v_2(K)+b_2) + v_3(K)\frac{b_2}{q^{2}}\bigr) \Donetwod \Dtwotwod+ v_5(K)\Dtwofourd\\
& &+ \Bigl\{\frac{1}{6}\bigl(v_2(K)+b_{2}\bigr)^{3}+ \frac{1}{2}(v_2(K)+b_2)^{2}\frac{b_2}{q^{2}}+(v_{2}(K)+b_{2})\frac{1}{2}\frac{b_2^2}{q^{4}}+ \frac{1}{6}b_{2}^{3}q^{-6}\Bigr\}\Donetwod \Donetwod \Donetwod \\
& & + \Bigl\{ \bigl( v_2(K)+b_2\bigr)\bigl( v_4(K)+b_4 \bigr)+ (v_{4}(K)+b_{4})\frac{b_2}{q^{2}} + (v_{2}(K)+b_2)\frac{b_4}{q^{4}}+\frac{ b_{2}b_{4}}{q^{6}}\Bigr\}\Donetwod\Donefourd \\
& & +\bigl(v_6(K) +b_{6}+\frac{b_{6}}{q^{6}}\bigr)\Donesixd + \bigr(-2b_2 \bigr)\bigl(v_2(K)+ b_2\bigr)\Atwod + (-2b_2v_3(K)){\raisebox{-2mm}{\includegraphics*[width=4mm]{A_3.eps}}} + (\mbox{other parts})
\end{eqnarray*}
By direct computations (or, one use $\mathfrak{sl}_2$ weight system evaluation as we will use in Section \ref{section:sl2computation}), the pairing with struts (in $\mathcal{A}^{red}$) are given by
\[ \begin{cases}
\langle \Dtwotwod \pairingcommaa \strutd\rangle = 2 \Atwod \,, \quad
\langle \Donetwod \Donetwod \pairingcommaa \strutd^2 \rangle = 16 \Atwod \,,\quad
\langle \Donefourd \pairingcommaa \strutd^2 \rangle = 20 \Atwod \,,\\
\langle \Dthreetwod \pairingcommaa \strutd^2 \rangle = 2 \Athreed \,, \quad
\langle \Dtwotwod \Donetwod \pairingcommaa \strutd^2 \rangle = 16 \Athreed \,, \quad
\langle \Dtwofourd \pairingcommaa \strutd^2 \rangle = 20 \Athreed \,, \\
\langle \Donetwod\Donetwod\Donetwod \pairingcommaa \strutd^{3}\rangle = 384\Athreed \,, \quad 
\langle \Donetwod\Donefourd \pairingcommaa \strutd^{3}\rangle= 480 \Athreed \,,\\
\langle \Donesixd \pairingcommaa \strutd^{3}\rangle = 420 \Athreed.
\end{cases}
\]
Consequently, we get
\begin{eqnarray*}
\lambda_2(S^{3}_{K}(p/q)) &=& v_{3}(K)\bigl(-\frac{q}{2p}\bigr)\cdot 2 \\
\nonumber & &+ \left( \frac{(v_2(K)+b_{2})^{2}}{2} + \frac{ (v_2(K)+b_2)b_2}{q^{2}} + \frac{b_2^{2}}{2q^{4}} \right)\frac{1}{2}\bigl(-\frac{q}{2p}\bigr)^{2}\cdot 16\\
\nonumber & &+\left( v_{4}(K)+b_{4}+\frac{b_{4}}{q^{4}} \right) \frac{1}{2}\bigl(-\frac{q}{2p}\bigr)^{2}\cdot 20 + \bigr(-2b_2 \bigr)\bigl(v_2(K)+ b_2\bigr).
\end{eqnarray*}
\begin{eqnarray*}
\lambda_3(S^{3}(K,p/q)) &\!\! = \!\!&
\bigl(w_4(K)-4b_2\bigl(v_2(K) +b_{2}\bigr)^{2}-10b_2v_4(K)\bigr)\bigl( -\frac{q}{2p}\bigr) \cdot 2 \\
& & \hspace{-1.5cm}+ \left( (v_3(K)(v_2(K)+b_2) + \frac{v_3(K)b_2}{q^{2}}\right)\frac{1}{2}\bigl( -\frac{q}{2p}\bigr)^{2} \cdot 16 + v_{5}(K)\frac{1}{2}\bigl( -\frac{q}{2p}\bigr)^{2}\cdot 20\\ 
& &\hspace{-1.5cm} + \left( \frac{\bigl(v_2(K)+b_{2}\bigr)^{3}}{6}+ \frac{(v_2(K)+b_2)^{2}b_{2}}{2q^{2}}+\frac{(v_{2}(K)+b_{2})b_{2}^{2}}{2q^{4}}+ \frac{b_{2}^{3}}{6q^{6}}\right) \frac{1}{6}\bigl( -\frac{q}{2p}\bigr)^{3}\cdot 384\\ 
& &\hspace{-1.5cm} + \left( \bigl( v_2+b_2\bigr)\bigl( v_4+b_4 \bigr)+ \frac{(v_{4}(K)+b_{4})b_2}{q^{2}} + \frac{(v_{2}(K)+b_2)b_{4}}{q^{4}}+ \frac{b_{2}b_{4}}{q^{6}} \right) \frac{1}{6}\bigl( -\frac{q}{2p}\bigr)^{3}\!\!\cdot480\\
& &\hspace{-1.5cm} + \left(v_6(K)+b_{6}+\frac{b_{6}}{q^{6}}\right)\frac{1}{6}\bigl( -\frac{q}{2p}\bigr)^{3}\cdot 420 +  (-2b_2v_3(K)).
\end{eqnarray*}
Thus we conclude
\begin{align*}
\lambda_2(S^{3}(K,p/q))-\lambda_{2}(L(p,q)) & \\
& \hspace{-3cm} = 
\left( v_2(K)^{2} + \frac{1}{24}v_2(K)+\frac{5}{2}v_{4}(K)\right)\frac{q^{2}}{p^{2}}-v_3(K)\frac{q}{p} + \frac{v_2(K)}{24}\left( \frac{1}{p^{2}}-1\right) \\
& \hspace{-3cm} =
\left( \frac{7a_2(K)^2-a_2(K)-10a_{4}(K)}{8} \right)\frac{q^{2}}{p^{2}} -v_{3}(K)\frac{q}{p} + \frac{a_{2}(K)}{48}
\left(1-\frac{1}{p^{2}}\right),
\end{align*}
\begin{align*}
\lambda_{3}(S^{3}(K,p/q))-\lambda_{3}(L(p,q)) & \\
& \hspace{-4cm} = 
-\left(\frac{35}{4}v_6(K)+\frac{5}{24}v_4(K)+10v_2(K)v_4(K)+ \frac{4}{3}v_2(K)^{3}+\frac{1}{12}v_2(K)^{2} 
\right)\frac{q^{3}}{p^{3}}\\
& \hspace{-3.5cm}
-\left( \frac{5}{24}v_4(K) +\frac{1}{288}v_2(K)+\frac{1}{12}v_2(K)^{2}\right)\frac{q}{p^{3}}\\
& \hspace{-3.5cm}
+ \left(\frac{5}{2}v_{5}(K)+ 2v_3(K)v_2(K)+\frac{1}{24}v_3(K)\right)\frac{q^2}{p^2} + \frac{v_3(K)}{24}\left(\frac{1}{p^2}-1\right)\\
& 
\hspace{-3.5cm}- \left(w_4(K)-\frac{1}{12}v_2(K)^{2}-\frac{1}{288}v_2(K)-\frac{5}{24}v_4(K)\right)\frac{q}{p}
\end{align*}

\end{proof}

\begin{remark}
\label{remark:KKT}
In \cite[Theorem7.1]{le} Lescop proved a similar formula 
\begin{equation}
\label{eqn:Lescopformula}
\lambda_2^{KKT}(S^{3}(K,p\slash q)) = \lambda^{'' KKT}_{2}(K) \frac{q^{2}}{p^{2}} + w_3(K)\frac{q}{p} + c(p\slash q)a_{2}(K) + \lambda_{2}^{KKT}(L(p,q))
\end{equation}
for the degree two part $\lambda_2^{KKT}$ of the Kontsevich-Kuperberg-Thurston universal finite type invariant $Z^{KKT}$, which is defined by configuration space integrals \cite{ko2,kt}.
For degree two part we have $\lambda_{2}^{KKT} = 2 \lambda_2$ (note that in Lescop uses the coefficient of the Jacobi diagram ${\raisebox{-2mm}{\includegraphics*[width=5mm]{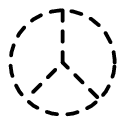}}} = \frac{1}{2}{\raisebox{-2mm}{\includegraphics*[width=4mm]{A2.eps}}}$) so Theorem \ref{theorem:LMO23} gives formulae of invariants in Lescop's formula (\ref{eqn:Lescopformula}), namely, 
\[ \lambda^{'' KKT}_{2}(K)=  \frac{7a_2(K)^2-a_2(K)-10a_{4}(K)}{4}, w_3(K)=-2v_3(K), c(p\slash q) = \frac{1}{24}-\frac{1}{24p^{2}}.\]
\end{remark}

\begin{proof}[Proof of Corollary \ref{corollary:LMO2}]
{$ $}\\

\noindent
(i,ii): 
By (\ref{eqn:NW}), it is sufficient to consider the case $r'= - r$.
Assume that $S^{3}(K,p/q) \cong \pm S^{3}(K,-p/q)$. Since $\lambda_2(M)=\lambda_2(-M)$, by Theorem \ref{theorem:LMO23} 
\[ \lambda_{2}(S^{3}(K,p/q)) - \lambda_{2}(\pm S^{3}(K,-p/q)) =  -2 v_{3}(K) q\slash p =0 \]
Therefore $v_{3}(K)=0$.\\

\noindent
(iii): Assume that $S^{3}(K,p/q) \cong \pm S^{3}(K, -p'/q')$. 
Since $H_{1}(S^{3}(K,p/q))\cong \Z\slash p\Z$, $p=p'$.
By Theorem \ref{theorem:LMO23} 
\begin{eqnarray*}
0 &= & \lambda_{2}(S^{3}(K,p/q)) - \lambda_{2}(- S^{3}(K,p/q')) \\
& =& \left(\frac{7a_{2}(K)^{2}-a_{2}(K)-10a_{4}(K)}{8}\right)\frac{q^{2}-q'^{2}}{p^{2}} - v_{3}(K)\frac{q-q'}{p}.
\end{eqnarray*}
Since $q' \neq \pm q$, either $v_{3}(K)\neq 0$ (and $7a_{2}(K)^{2}-a_{2}(K)-10a_{4}(K)=0$), or, 
\[ \frac{p}{q+q'} = \frac{7a_{2}(K)^{2}-a_{2}(K)-10a_{4}(K)}{8v_{3}(K)}.\]

\noindent
(iv): Assume that $S^{3}(K,p/q) \cong \pm S^{3}(K', p/q)$. By Theorem \ref{theorem:LMO1} (i), $a_{2}(K)=a_{2}(K')$. By Theorem \ref{theorem:LMO23} 
\[\lambda_{2}(S^{3}(K,p/q)) - \lambda_{2}(S^{3}(K',p/q)) = \frac{5}{4}(a_{4}(K)-a_{4}(K'))\frac{q^{2}}{p^{2}} - (v_{3}(K)-v_{3}(K'))\frac{q}{p}=0 .\]
\end{proof}

\begin{proof}[Proof of Corollary \ref{corollary:Lens}]
Assume that $S^{3}(K,p\slash q) \cong L(p',q')$. Then $|H_{1}(S^{3}(K,p\slash q);\Z)|=p=p'=H_{1}(L(p',q));\Z)$ so $p=p'$. By (\ref{eqn:LMO23-Lens}) $\lambda_{2}(L(p,q))=\lambda_{2}(L(p,q'))$ so Theorem \ref{theorem:LMO23} gives the desired equality.
\end{proof}

\begin{proof}[Proof of Corollary \ref{corollary:LMO3}]
(i): By (\ref{eqn:NW}), it is sufficient to consider the case $r'= -r$.
By Theorem \ref{theorem:LMO1} (ii) and Corollary \ref{corollary:LMO2} (i), $a_{2}(K)=v_2(K)=v_{3}(K)=0$. Thus by Theorem \ref{theorem:LMO23} 
\begin{eqnarray*}
 0 &=& \lambda_{3}(S^{3}(K,p/q)) - \lambda_{3}(S^{3}(K,-p/q))\\
& =& -\left( \frac{35}{2}v_{6}(K) + \frac{5}{12} v_4(K)\right) \frac{q^{3}}{p^{3}} - \frac{5}{12}v_{4}(K) \frac{q}{p^{3}} - \left( 2w_{4}(K)-\frac{5}{12}v_{4}(K)\right) \frac{q}{p}.
\end{eqnarray*}
(ii):
If $S^{3}_{K}(p/q)) \cong -S^{3}_{K}(-p/q)$ then by Corollary \ref{corollary:LMO2} (ii) $v_{3}(K)=0$. Therefore by Theorem
$ \lambda_{3}(S^{3}(K,p/q)) - \lambda_{3}(-S^{3}(K,-p/q)) = 5v_5(K)= 0$. 
\end{proof}

\begin{proof}[Proof of Corollary \ref{cor:c11}]

By Lemma \ref{lemma:AJ} and Corollary \ref{corollary:LMO3} (i), if $S^{3}(K,p\slash q) \cong S^{3}(K,-p\slash q)$ then we get
\begin{equation}
\label{eqn:obstruction}
(19a_{4}(K)+j_4(K))p^{2} -10a_{4}(K)-(420a_{6}(K)+80a_{4}(K))q^{2}=0.
\end{equation}

According to \cite{iw}, the cosmetic surgery conjecture was confirmed for knots with less than or equal to 11 crossings, with 8 exceptions
\[ 10_{33},10_{118},10_{146},11a_{91},11a_{138},11a_{285},11n_{86},11n_{157}
\]
in the table KnotInfo \cite{cl}. 

For these knots, the values of $a_{4},j_{4}$ and $a_{6}$ are given as follows.

\begin{table}[htbp]
\begin{tabular}{|c|c|c|c|c|c|c|c|c|} \hline
& $10_{33}$ & $10_{118}$ & $10_{146}$ & $11a_{91}$ & $11a_{138}$ &  $11a_{285}$ & $11n_{86}$ & $11n_{157}$ \\ \hline 
$a_{4}$ &
4 & 2 & 2 & 0 & 2 & 2& -2& 0     \\ \hline
$j_{4}$ &-12 &-6 & -6 & 0& -6 & -6& 6& 0\\ \hline
$a_{6}$ &0& 3 & 0 & -2& -2& -2& -1& -1\\ \hline
\end{tabular}
\end{table}

Note that (\ref{eqn:obstruction}) gives a diophantine equation of the form $ap^{2}-bq^{2}=c$, whose solvability can be checked algorithmically \cite{aa}. For these knots the equation (\ref{eqn:obstruction}) has no integer solutions, except the case $K=10_{118}$ (The author uses the computer program at \cite{so}. In the case $K=10_{118}$ we get the equation $32p^2-20-1420q^{2}=0$ which has the solutions $p=20u+1065v$ and $q=3u-160v$, where $(u,v)$ are the solutions of Pell's equation $u^2 - 2840v^2 = 1$.)

\end{proof}

Next we proceed to see higher degree part. To use $C_{n}$-equvalence assumption, we observe the following.

\begin{lemma}
\label{lemma:KontCn}
If $K$ and $K'$ are $C_{n+1}$-equivalent, then for $e+\frac{k}{2} \leq n+1$, then 
\[ (Z^{\sf Wheel}(K) \sqcup \Omega_{q} - Z^{\sf Wheel}(K) \sqcup \Omega_{q})_{e,k} = (Z^{\sigma}(K) \sqcup \Omega^{-1})_{e,k} - (Z^{\sigma}(K') \sqcup \Omega^{-1})_{e,k} \]
Similarly, if  $K$ and $K'$ are odd $C_{n+1}$-equivalent, then for $e+\frac{k}{2} \leq n+1$ with odd $e+k$, 
\[ (Z^{\sf Wheel}(K) \sqcup \Omega_{q} - Z^{\sf Wheel}(K) \sqcup \Omega_{q})_{e,k} = (Z^{\sigma}(K) \sqcup \Omega^{-1})_{e,k} - (Z^{\sigma}(K') \sqcup \Omega^{-1})_{e,k} \]
\end{lemma}
\begin{proof}
Since $K$ and $K'$ are $C_{n}$ equivalent 
$Z^{\sigma}(K)_{e',k'} = Z^{\sigma}(K')_{e',k'}$ if $e'+k' \leq n$.
Since for the degree $d$ element $D$, the $\partial_{\Omega^{-1}}(D) = D+ (\mbox{degree} \geq d+2\mbox{ elements } )$, for $e+k \leq n+1$ 
\[
Z^{\sf Wheel}(K)_{e,k}-Z^{\sf Wheel}(K')_{e,k} = (\partial_{\Omega^{-1}}(Z^{\sigma}(K)-Z^{\sigma}(K'))_{e,k} =
(Z^{\sigma}(K)-Z^{\sigma}(K'))_{e,k}. \]
Therefore for $e+k \leq n+1$
\begin{eqnarray*}
(Z^{\sf Wheel}(K) \sqcup \Omega_{q} - Z^{\sf Wheel}(K') \sqcup \Omega_{q})_{e,k}& = &  \left( (Z^{\sf Wheel}(K) - Z^{\sf Wheel}(K')) \sqcup \Omega_{q}\right)_{e,k}\\
& = & \left( (Z^{\sigma}(K)-Z^{\sigma}(K'))\sqcup \Omega_q \right)_{e,k}\\
& = & (Z^{\sigma}(K)-Z^{\sigma}(K'))_{e,k}\\
& = & (Z^{\sigma}(K) \sqcup \Omega^{-1})_{e,k} - (Z^{\sigma}(K') \sqcup \Omega^{-1})_{e,k} 
\end{eqnarray*}

To see the latter assertion, we note that both $\partial_{\Omega^{-1}}$ and $\sqcup \Omega_{q}$ preserve the parity of $D$. Namely, 
 if $D \in \mathcal{B}_{odd}$, where we denote by $\mathcal{B}_{odd}$ the odd degree part of $\mathcal{B}$, then $\partial_{\Omega^{-1}}(D), D\sqcup \Omega_{q} \in \mathcal{B}_{odd}$.
\end{proof}

\begin{proof}[Proof of Theorem \ref{theorem:LMOhigh1}]
{$ $}\\

\noindent
(i): 
By Proposition \ref{prop:degree}, the degree $m+1$ part of the LMO invariant of $S^{3}(K,p\slash q)$ is determined by $(Z^{\sf Wheel}(K) \sqcup \Omega_{q})_{e,k}$ with $e+\frac{k}{2}=m+1$.

Since $K$ and $K'$ are $C_{2m+2}$-equivalent, for $e+\frac{k}{2}=m+1$
\begin{eqnarray*}
(Z^{\sf Wheel}(K) \sqcup \Omega_{q} - Z^{\sf Wheel}(K') \sqcup \Omega_{q})_{e,k} &=&(Z^{\sigma}(K)\sqcup \Omega^{-1})_{e,k} - (Z^{\sigma}(K')\sqcup \Omega^{-1})_{e,k}\\
& & \hspace{-2cm}=
 \begin{cases}
 -\frac{1}{2}(a_{2m+2}(K)-a_{2m+2}(K'))\overbrace{\raisebox{-4.5mm}{\includegraphics*[width=9mm]{D_12m.eps}}}^{2m+2} & (e,k)=(0,2m+2)\\
\quad 0 & \mbox{Otherwise},
\end{cases}
\end{eqnarray*}
by Lemma \ref{lemma:KontCn} and Proposition \ref{proposition:MMR}. Therefore
\begin{align}
\label{eqn:LMO_2m+1}
Z^{LMO}_{m+1}(S^{3}(K,p\slash q))-Z^{LMO}_{m+1}(S^{3}(K',p\slash q)) & \\ 
&  \nonumber \hspace{-6cm} = \left\langle -\frac{1}{2}(a_{2m+2}(K)-a_{2m+2}(K'))\overbrace{\raisebox{-4.5mm}{\includegraphics*[width=9mm]{D_12m.eps}}}^{2m+2} \pairingcommab \frac{1}{(m+1)!}\left(-\frac{q}{2p}\right)^{2m+1} \!\!\!\!\! \strutd^{m+1} \right\rangle \\
 \nonumber&\ \hspace{-6cm} = - \frac{a_{2m+2}(K)-a_{2m+2}(K')}{2(m+1)!}\left(-\frac{q}{2p}\right)^{m+1}
\left\langle \overbrace{\raisebox{-4.5mm}{\includegraphics*[width=9mm]{D_12m.eps}}}^{2m+2} \pairingcommab \strutd^{m+1} \right\rangle  \pairingperiod
\end{align}

As we will see in Lemma \ref{lemma:nonvanish},
$\left\langle \overbrace{\raisebox{-4.5mm}{\includegraphics*[width=9mm]{D_12m.eps}}}^{2m+2}  \pairingcommab \strutd^{m+1} \right\rangle \neq 0$.
This shows that $S^{3}(K,p\slash q) \cong S^{3}(K',p\slash q)$ implies $a_{2m+2}(K)=a_{2m+2}(K')$.\\

\noindent
(ii) By (\ref{eqn:LMO_2m+1}), when $K$ is $C_{4m+2}$-trivial, then
\[ 
Z^{LMO}_{2m+1}(S^{3}(K,p\slash q))-Z^{LMO}_{2m+1}(L(p,q)) = \left\langle -\frac{1}{2}a_{4m+2}(K)\overbrace{\raisebox{-4.5mm}{\includegraphics*[width=9mm]{D_12m.eps}}}^{4m+2} \pairingcommab \frac{1}{(2m+1)!}\left(-\frac{q}{2p}\right)^{2m+1} \!\!\!\!\!\! \strutd^{2m+1} \right\rangle \pairingperiod 
\]
By (\ref{eqn:NW}), if $S^{3}(K,p\slash q) \cong S^{3}(K,p'\slash q')$ then 
$\frac{p}{q}= - \frac{p'}{q'}$ and $Z^{LMO}_{2m+1}(L(p,q))-Z^{LMO}(L(p',q'))=0$ hence
\begin{eqnarray*}
0 &=& Z^{LMO}_{2m+1}(S^{3}(K,p\slash q))-Z^{LMO}_{2m+1}(S^{3}(K,-p\slash q))\\
 & =&  - \frac{a_{4m+2}(K)}{(2m+1)!}\left(-\frac{q}{2p}\right)^{2m+1} \left\langle\overbrace{\raisebox{-4.5mm}{\includegraphics*[width=9mm]{D_12m.eps}}}^{4m+2} \pairingcommab \strutd^{2m+1} \right\rangle \pairingperiod
\end{eqnarray*}
Therefore $a_{4m+2}(K)=0$.
\end{proof}

\begin{proof}[Proof of Theorem \ref{theorem:LMOhigh2}]
{$ $}\\

\noindent
(i): By the same argument as Theorem \ref{theorem:LMOhigh1} (i), if $K$ and $K'$ are $C_{2m+1}$-equivalent for $e+\frac{k}{2}=m+1$,
\begin{eqnarray*}
(Z^{\sf Wheel}(K) \sqcup \Omega_{q} - Z^{\sf Wheel}(K') \sqcup \Omega_{q})_{e,k} &=& (Z^{\sigma}(K)\sqcup \Omega^{-1})_{e,k} - (Z^{\sigma}(K')\sqcup \Omega^{-1})_{e,k}\\
& & \hspace{-3cm}= 
 \begin{cases}
 -\frac{1}{2}(a_{2m+2}(K)-a_{2m+2}(K'))\overbrace{\raisebox{-4.5mm}{\includegraphics*[width=9mm]{D_12m.eps}}}^{2m+2} & (e,k)=(0,2m+2)\\
(Z^{\sigma}(K) \sqcup \Omega^{-1} - Z^{\sigma}(K) \sqcup \Omega^{-1})_{1,2m} &  (e,k)=(1,2m)\\
0 & \mbox{Otherwise.}
\end{cases}
\end{eqnarray*}
Hence
\begin{eqnarray*}
Z^{LMO}_{m+1}(S^{3}(K,p\slash q))- Z^{LMO}_{m+1}(S^{3}(K',p\slash q)) & & \\
& &\hspace{-6cm} = \left\langle -\frac{1}{2}(a_{2m+2}(K)-a_{2m+2}(K'))\overbrace{\raisebox{-4.5mm}{\includegraphics*[width=9mm]{D_12m.eps}}}^{2m+2}
\pairingcommab \frac{1}{(m+1)!}\left(-\frac{q}{2p}\right)^{m+1} \!\!\!\! \strutd^{m+1} \right\rangle \\
& & \hspace{-5cm}+ \left\langle  (Z^{\sigma}(K)\sqcup \Omega^{-1})_{1,2m} - (Z^{\sigma}(K')\sqcup \Omega^{-1})_{1,2m} \pairingcommab \frac{1}{m!}\left(-\frac{q}{2p}\right)^{m}\!\!\!\! \strutd^{m} \right\rangle \pairingperiod
\end{eqnarray*}
Thus, if $S^{3}(K,p\slash q) \cong S^{3}(K', p\slash q)$ and $a_{4m+4}(K)=a_{4m+4}(K')$ then
\[ \left\langle \Bigl( (Z^{\sigma}(K)\sqcup \Omega^{-1})_{1,4m+2} - (Z^{\sigma}(K')\sqcup \Omega^{-1})_{1,4m+2} \Bigr), \strutd^{2m+1} \right\rangle =0.
\]

\noindent
(ii,iii) Assume that $K$ is odd $C_{4m+1}$-trivial.
Since
\begin{align*}
Z^{LMO}_{2m+1}(S^{3}(K,p\slash q)) = \sum_{e=0}^{m} \left\langle (Z^{\sf Wheel}(K)\sqcup \Omega_{q})_{2e,4m+2-4e} \pairingcommab \frac{1}{(2m+1-2e)!} \left( -\frac{q}{2p}\strutd \right)^{2m+1-2e} \right\rangle \\
+ \sum_{e=0}^{m} \left\langle (Z^{\sf Wheel}(K)\sqcup \Omega_{q})_{2e+1,4m-4e} \pairingcommab\frac{1}{(2m-2e)!} \left( -\frac{q}{2p}\strutd \right)^{2m-2e} \right\rangle
\end{align*}
we get
\begin{align*}
&0= Z^{LMO}_{2m+1}(S^{3}(K,p\slash q))-Z^{LMO}_{2m+1}(-S^{3}(K,-p\slash q)) = Z^{LMO}_{2m+1}(S^{3}(K,p\slash q)) + Z^{LMO}_{2m+1}(S^{3}(K,-p\slash q))\\
& \hspace{0.5cm}= 2 \sum_{e=0}^{m} \left\langle (Z^{\sf Wheel}(K)\sqcup \Omega_{q})_{2e+1,4m-4e} \pairingcommab \frac{1}{(2m-2e)!} \left( -\frac{q}{2p}\strutd \right)^{2m-2e} \right\rangle \pairingperiod
\end{align*}
By Lemma \ref{lemma:KontCn}, $(Z^{\sf Wheel}(K)\sqcup \Omega_{q})_{2e+1,4m-4e} = 0$ unless $e=0$. Moreover for $e=0$ we have
\[
(Z^{\sf Wheel}(K)\sqcup \Omega_{q})_{1,4m}=(Z^{\sf Wheel}({\sf Unknot})\sqcup \Omega_q)_{1,4m} + (Z^{\sigma}(K)\sqcup \Omega^{-1})_{1,4m} - (Z^{\sigma}({\sf Unknot})\sqcup \Omega^{-1})_{1,4m}.
\]
Since for any $X \in \mathcal{B}$ and $q \in \Z \setminus\{0\}$, $(X\sqcup \Omega_q^{\pm 1})_{1,k}= X_{1,k}$ we have
\[ (Z^{\sf Wheel}({\sf Unknot})\sqcup \Omega_q)_{1,4m} = (Z^{\sigma}({\sf Unknot})\sqcup \Omega^{-1})_{1,4m}.\]
Thus we get
\[ (Z^{\sf Wheel}(K)\sqcup \Omega_{q})_{1,4m} = (Z^{\sigma}(K)\sqcup \Omega^{-1})_{1,4m}.\]
Hence
\[ 0= \frac{2}{(2m)!}\left(-\frac{q}{2p}\right)^{2m} \!\!\left\langle (Z^{\sigma}(K)\sqcup \Omega^{-1})_{1,4m} \pairingcommaa \strutd^{2m} \right\rangle .
\]
(iii) is proved similarly.

\end{proof}

\begin{proof}[Proof of Corollary \ref{corollary:MM}]
Lemma \ref{lemma:nonvanish2} in next Section shows that if $(Z^{\sigma}(K)\sqcup \Omega^{-1})_{1,2m} \in K_{1,2m}$ then $j_{1,2m}(K)=0$.
\end{proof}

\section{Some $\mathfrak{sl}_2$ weight system computations}
\label{section:sl2computation}

In this section we use $(\mathfrak{sl}_2,V_n)$ weight system, which is a linear map $W_{(\mathfrak{sl_2},V_n)} : \mathcal{B} \ (\mbox{or, } \mathcal{A}(\emptyset))\rightarrow \C[[h]]$ that comes from the Lie algebra $\mathfrak{sl}_2$ and its $n$-dimensional irreducible representation, to confirm some assetions used in previous sections.

The image of $W_{\mathfrak{sl}_2,V_n}$ can be calculated recursively by the following relations \cite{cv}:
\begin{enumerate}
\item $W_{\mathfrak{sl_2}}(\raisebox{-2.5mm}{\includegraphics*[width=7mm]{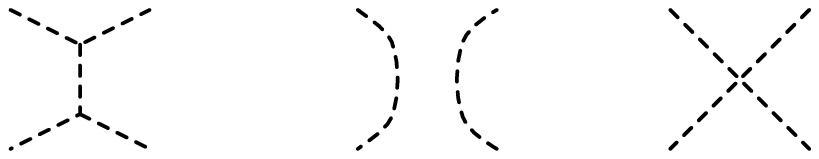}}) = 2h\Bigr(W_{\mathfrak{sl_2}}(\raisebox{-2.5mm}{\includegraphics*[width=7mm]{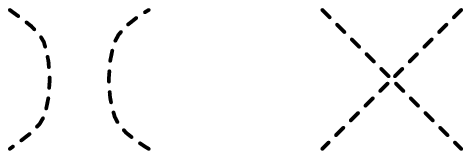}})- W(\raisebox{-2.5mm}{\includegraphics*[width=7mm]{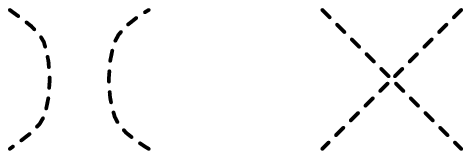}})\Bigl)$ 
\item $W_{\mathfrak{sl_2}}(\raisebox{-2.5mm}{\includegraphics*[width=7mm]{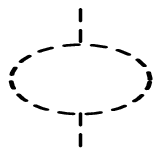}}) = 4h W_{\mathfrak{sl_2}}(\raisebox{-2.5mm}{\includegraphics*[width=7mm]{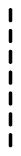}})$.
\item $W_{\mathfrak{sl_2}}(\raisebox{-3mm}{\includegraphics*[width=7mm]{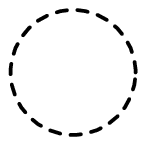}})=3$.
\item $W_{(\mathfrak{sl_2},V_n)}(D \sqcup \raisebox{-1mm}{\includegraphics*[width=7mm]{strut.eps}}) = h \frac{n^{2}-1}{2}W_{(\mathfrak{sl_2},V_n)}(D)$.
\end{enumerate}
%Here $C$ denotes the Casimir element, whose eigenvalue on $V_{n}$ is $\frac{n^{2}-1}{2}$.
Note that $\deg W_{(\mathfrak{sl_2},V_n)}(D) = \deg(D)$ and the relation (1)--(3) do not depend on $n$.  Thus for $D \in \mathcal{A}(\emptyset)$, $ W_{(\mathfrak{sl_2},V_n)}(D)$ does not depend on $n$ so we will simply write by $W_{\mathfrak{sl_2}}(D)$.

By the definition of the weight system and the colored Jones polynomial as quantum $(\mathfrak{sl}_2, V_n)$ invariant, we have
\begin{equation}
\label{eqn:Jones}
 W_{(\mathfrak{sl_2},V_n)}(Z^{\sigma}(K)\sqcup \Omega^{-1}) = V_{n}(K;e^{-h})= \sum_{e\geq 0}  \left(\sum_{k\geq 0} j_{e,k}(K) (nh)^{k}\right)h^{e}.
\end{equation}
We remark that we put the variable $t$ in the colored Jones polynomial not $e^{h}$ but $e^{-h}$, due to the difference of normalization of the colored Jones polynomial and quantum $\mathfrak{sl}_2$ invariants. 

We use this to check finite type invariants which we used can be written by Jones and Conway polynomials.

\begin{proof}[Proof of Lemma \ref{lemma:AJ}]
(1),(3) and (5) follow from Proposition \ref{proposition:MMR} so we prove (2) and (4).
The degree three and four part of $(Z^{\sigma}(K)\sqcup \Omega^{-1})$ are given $v_{3}(K)\Dtwotwod$ and $\frac{1}{8}a_{2}(K)^{2}\Donetwod\Donetwod - \frac{1}{2}a_{4}(K) \Donefourd+ w_{4}(K)\Dthreetwod$, respectively. Thus by (\ref{eqn:Jones}) applying $W_{(\mathfrak{sl_2},V_2)}$ we get
\begin{eqnarray*}
j_{3}(K)(-h)^{3} & = & v_{3}(K) W_{(\mathfrak{sl}_2, V_2)}(\Dtwotwod)= 24v_{3}(K)h^{3} \\
 j_{4}(K)(-h)^{4} & = & \frac{1}{8}a_{2}(K)^{2} W_{(\mathfrak{sl}_2, V_2)}(\Donetwod\Donetwod) - \frac{1}{2}a_{4}(K) W_{(\mathfrak{sl}_2, V_2)}( \Donefourd)+ w_{4}(K) W_{(\mathfrak{sl}_2, V_2)}(\Dthreetwod)\\
& = &\frac{1}{8}a_{2}(K)^{2}36h^{4} - \frac{1}{2}a_{4}(K) 18h^{4}+ w_{4}(K)96h^{4}\\
& = & \left(\frac{9}{2}a_{2}(K)^{2} -9a_{4}(K) + 96w_{4}(K)\right) h^{4}
\end{eqnarray*}
\end{proof}

For two Jacobi diagrams $D$ and $D'$ we write $D \equiv D'$ if $D$ is equal to $D'$ by using the $\mathfrak{sl}_2$ weight system relations (1)--(3) which are independent of $V_n$. 
By the $\mathfrak{sl}_2$ weight system relations (1)--(3) we remove all trivalent vertex of a Jacobi diagram when the number of univalent vertices are even. 
% If $D \equiv e \raisebox{-1mm}{\includegraphics*[width=7mm]{strut.eps}}^{n}, D' \equiv e' \raisebox{-1mm}{\includegraphics*[width=7mm]{strut.eps}}^{n}$ then (see \cite[Lemma 6.1]{blt})
%\[
% W_{\mathfrak{sl}_2}\left(\langle D, D' \rangle \right) = ee' W_{\mathfrak{sl}_2}\left( \Bigl\langle\raisebox{-1mm}{\includegraphics*[width=7mm]{strut.eps}}^{n}, \raisebox{-1mm}{\includegraphics*[width=7mm]{strut.eps}}^{n} \Bigr\rangle \right) = ee'(2n+1)!.
%\]

\begin{lemma}
\label{lemma:nonvanish}
$W_{\mathfrak{sl}_2} \Bigl(\Bigl\langle \overbrace{\raisebox{-4mm}{\includegraphics*[width=10mm]{D_12m.eps}}}^{2m} \pairingcommab \raisebox{-1mm}{\includegraphics*[width=7mm]{strut.eps}}^m \Bigr\rangle \Bigr) = 2(2h)^{m}(2m+1)!$. In particular,
 $\Bigl\langle \overbrace{\raisebox{-4mm}{\includegraphics*[width=10mm]{D_12m.eps}}}^{2m}  \pairingcommab \raisebox{-1mm}{\includegraphics*[width=7mm]{strut.eps}}^m \Bigr\rangle \neq 0$
\end{lemma}
\begin{proof}
First we observe that
\[
 \overbrace{\raisebox{-4mm}{\includegraphics*[width=10mm]{D_12m.eps}}}^{2m} \equiv (2h)^{m}\sum_{\be=(e_1,\ldots,e_m) \in \{0,1\}^{m}} (-1)^{e_1+\cdots+e_m} D_{\be} \\ 
\]
Here for $\be=(e_1,\ldots,e_m) \in \{0,1\}^{m}$, $D_{\be}$ denotes the Jacobi diagram\\
\[ \includegraphics*[width=70mm]{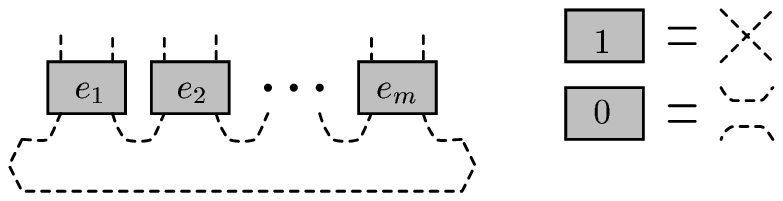}\]
Then the pairing $ \left\langle D_{\be}, \raisebox{-1mm}{\includegraphics*[width=7mm]{strut.eps}}^m \right\rangle = {\raisebox{-5mm}{\includegraphics*[width=20mm]{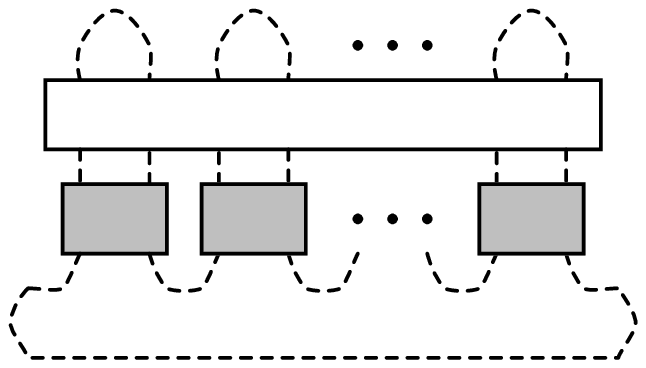}}}$
 is given by
\begin{equation}
\label{eqn:pairing0}
 {\raisebox{-5mm}{\includegraphics*[width=25mm]{d1.eps}}} = \begin{cases} \includegraphics*[width=30mm]{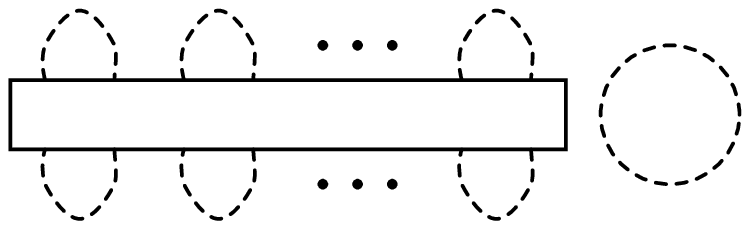} & (e_{1}=e_2=\cdots=e_m=0)\\
\includegraphics*[width=25mm]{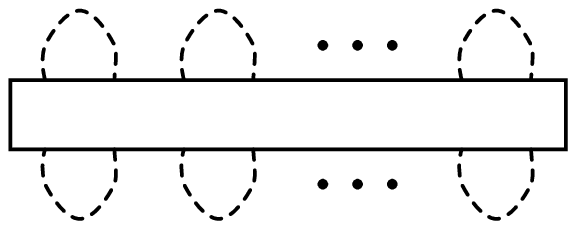} & (\text{otherwise})
\end{cases}
\end{equation}
By definition, $W_{\mathfrak{sl}_2}\left({\raisebox{-3mm}{\includegraphics*[width=20mm]{d2.eps}}}\right)= W_{\mathfrak{sl}_2}\left(\Bigl\langle\raisebox{-1mm}{\includegraphics*[width=7mm]{strut.eps}}^{m}, \raisebox{-1mm}{\includegraphics*[width=7mm]{strut.eps}}^{m} \Bigr\rangle \right) = (2m+1)!$ (see \cite[Lemma 6.1]{blt}).
Therefore by (\ref{eqn:pairing0})
\[
W_{\mathfrak{sl}_2} \Bigl(\Bigl\langle \overbrace{\raisebox{-4mm}{\includegraphics*[width=10mm]{D_12m.eps}}}^{2n}{ \raisebox{-2mm},} \raisebox{-1mm}{\includegraphics*[width=7mm]{strut.eps}}^m \Bigr\rangle \Bigr) = 2(2h)^{m}(2m+1)!
\]
\end{proof}

\begin{lemma}
\label{lemma:nonvanish2}
\[ W_{\mathfrak{sl_2}}\left(\left\langle (Z^{\sigma}(K)\sqcup \Omega^{-1})_{1,2m} {\raisebox{-2mm},} \
\raisebox{-1.8mm}{\includegraphics*[width=7mm]{strut.eps}}^{m}
\right\rangle \right) = 2^{m}h^{m+1} (2m+1)!  j_{1,2m}(K)\]
\end{lemma}
\begin{proof}
Let us put $(Z^{\sigma}(K)\sqcup \Omega^{-1})_{e,2k} \equiv  c_{e,2k}(K) h^{e+k} \raisebox{-1mm}{\includegraphics*[width=7mm]{strut.eps}}^{k}$. Then 
\begin{align*}
W_{\mathfrak{sl}_2,V_{n}}((Z^{\sigma}(K)\sqcup \Omega^{-1})_{e,2k}) &=  e_{e,2k}(K) h^{e+2k}\left(\frac{n^{2}-1}{2}\right)^{k} \\
&\hspace{-3.2cm} = \frac{c_{e,2k}(K)}{2^{k}} h^{e}(nh)^{2k} -\frac{c_{e,2k}(K)}{2^{k}}kh^{e+2}(nh)^{2k-2}+ \frac{c_{e,2k}(K)}{2^{k}}\binom{k}{2}h^{e+4}(nh)^{2k-4}-\cdots. \\
\end{align*}
By (\ref{eqn:Jones}) we conclude $j_{1,2m}(K) = \frac{c_{1,2m}(K)}{2^{m}}$.Then the $\mathfrak{sl}_2$ weight system evaluation of the desired pairing is 
\begin{eqnarray*} W_{\mathfrak{sl_2}}\left( \left\langle(Z^{\sigma}(K)\sqcup \Omega^{-1})_{1,2m} {\raisebox{-2mm},} 
\raisebox{-1.8mm}{\includegraphics*[width=7mm]{strut.eps}}^{m}
\right\rangle \right)& = & 2^{m}j_{1,2m}(K)h^{m+1}  W_{\mathfrak{sl_2}}\left({\raisebox{-3mm}{\includegraphics*[width=20mm]{d2.eps}}}\right)\\
& = & 2^{m}h^{m+1} (2m+1)!  j_{1,2m}(K)
\end{eqnarray*}
 \end{proof}

\end{document}